\documentclass[12pt]{amsart}
\usepackage{mathrsfs}

\usepackage{amssymb,amsmath,graphicx,color,textcomp, amsthm,bbm,bbold, enumerate,booktabs}
\usepackage{chngcntr}
\usepackage{apptools}
\AtAppendix{\counterwithin{theorem}{section}}
\definecolor{Red}{cmyk}{0,1,1,0}

\definecolor{verde}{cmyk}{1,0,1,0}

\definecolor{loka}{cmyk}{.5,0,1,.5}
\definecolor{azul}{cmyk}{1,1,0,0}

\numberwithin{equation}{section}

\newcommand{\be}{\begin{equation}}
\newcommand{\ee}{\end{equation}}

\newtheorem{theorem}{Theorem}

\newtheorem{definition}{Definition}

\newtheorem{corollary}[equation]{Corollary}

\begin{document}
\title{On the local $M$-derivative}
\author{J. Vanterler da C. Sousa$^1$}
\address{$^1$ Department of Applied Mathematics, Institute of Mathematics,
 Statistics and Scientific Computation, University of Campinas --
UNICAMP, rua S\'ergio Buarque de Holanda 651,
13083--859, Campinas SP, Brazil\newline
e-mail: {\itshape \texttt{ra160908@ime.unicamp.br, capelas@ime.unicamp.br }}}

\author{E. Capelas de Oliveira$^1$}

\begin{abstract} We introduce a new local derivative that generalizes the so-called 	alternative $"fractional"$ derivative recently proposed.  We denote this new differential operator by $\mathscr{D}_{M}^{\alpha,\beta}$, where the parameter $\alpha$, associated with the order, is such that $0<\alpha<1$, $\beta>0$ and $M$ is used to denote that the function to be derived involves a Mittag-Leffler function with one	parameter. 

This new derivative satisfies some properties of integer-order calculus, e.g. linearity, product rule, quotient rule, function composition and the chain rule. Besides as in the case of the Caputo derivative, the	derivative of a constant is zero. Because Mittag-Leffler function is a natural generalization of the exponential function, we can extend some of the classical results, namely: Rolle's theorem, the mean value theorem and its extension.

We present the corresponding $M$-integral from which, as a natural consequence, new results emerge which can be interpreted as applications. Specifically, we generalize the inversion property of the fundamental theorem of calculus and prove a theorem associated with	the classical integration by parts. Finally, we present an application involving linear differential equations by means of local $M$-derivative with some graphs.
\vskip.5cm
\noindent
\emph{Keywords}: Local $M$-Derivative, Local $M$-Differential Equation, $M$-Integral, Mittag-Leffler Function.
\newline 
MSC 2010 subject classifications. 26A06; 26A24; 26A33; 26A42.
\end{abstract}
\maketitle


\section{Introduction} 

The integral and differential calculus of integer-order developed by Leibniz and Newton was a great discovery in mathematics, having numerous applications in several areas of physics, biology, engineering and others.  But something intriguing and interesting to the mathematicians of the day was still to come. In 1695 \cite{GWL1,GWL2,GWL3}, $\ell$'Hospital, in a letter to Leibniz, asks
him about the possibility of extending the meaning of an integer-order derivative $d^{n}y/dx^{n}$ to the case in which the order is a fraction. This question initiated the history of a new calculus which was called non-integer order calculus and which nowadays is usually called fractional calculus.

Although fractional calculus emerged at the same time as the integer-order calculus proposed by Newton and Leibniz, it did not attract the attention of the scientific community and for many years remained hidden. It was only after an international congress in 1974 that fractional calculus began to be known and consolidated in numerous applications in several fields such as mathematics, physics, biology and engineering.

Several types of fractional derivatives have been introduced to date, among which the Riemann-Liouville, Caputo, Hadamard, Caputo-Hadamard, Riesz and other types \cite{ECJT}. Most of these derivatives are defined on the basis of the corresponding fractional integral in the Riemann-Liouville sense.

Recently, Khalil et al. \cite{RMAM} proposed the so-called conformable fractional derivative of order $\alpha$, $0<\alpha<1$, in order to generalize classical properties of integer-order calculus. Some applications of the conformable fractional derivative  and the alternative fractional derivative are gaining space in the field of the fractional calculus and numerous works using such derivatives are being published of which we mention: the heat equation, the Taylor formula and some inequalities of convex functions \cite{DRA,YC}. More recently, in 2014, Katugampola \cite{UNK} also proposed a new fractional derivative with classical properties, similar to the conformable fractional derivative.

Given such a variety of definitions we are naturally led to ask which are the criteria that must be satisfied by an operator, differential or integral, in order to be called a fractional operator. In 2014, Ortigueira and Machado \cite{JAM,MDJA} discussed the concepts underlying those definitions and pointed out some properties that, according to them, should be satisfied by such operators (derivatives and integrals) in order to be called fractional. However, there exist operators which one would like to call fractional \cite{RMAM,UNK} even though they do not satisfy the criteria proposed by Ortigueira and Machado, and this led Katugampola \cite{UNK1} to criticize those criteria. 

The main motivation for this work comes from the alternative fractional derivative recently introduced \cite{UNK} and some new results involving the conformable fractional derivative \cite{RMAM,OSEN,TA}, all of which constitute particular cases of our results. In this sense, as an application of local $M$-derivative, we present the general solution of a linear differential equation with graphs.

This paper is organized as follows: in section 2 we present the concepts of fractional derivatives in the Riemann-Liouville and Caputo sense, the definition of fractional derivative by Khalil et al. \cite{RMAM} and the alternative definition proposed by Katugampola \cite{UNK}, together with their
properties. In section 3, our main result, we introduce the concept of an local $M$-derivative involving a Mittag-Leffler function and demonstrate several theorems. In section 4 we introduce the corresponding $M$-integral, for which we also present several results; in particular, a generalization of the fundamental theorem of calculus. In section 5, we present the relation between the local $M$-derivatives, introduced here, and the alternative proposed in \cite{UNK}. In section 6, we present an application involving linear differential equations by means of local $M$-derivative with some graphs. Concluding remarks close the paper.


\section{Preliminaries}

The most explored and studied fractional derivatives of fractional calculus are
the so-called Riemann-Liouville and Caputo derivatives. Both types are
fundamental in the study of fractional differential equations; their definitions
are presented below.

\begin{definition} Let $\alpha\in\mathbb{C}$ such that $Re\left( \alpha \right)
	>0$ and $m-1<\alpha \leq m$. The fractional derivative of order $\alpha$
	of a causal function $f$ in the Riemann-Liouville sense,
	$\mathcal{D}_{RL}^{\alpha }f\left( t\right)$, is defined by
	{\rm\cite{IP,RM,RCEC}}
\begin{equation}\label{L7}
\mathcal{D}_{RL}^{\alpha }f\left( t\right) := D^{m}\textbf{J}^{m-\alpha }f\left( t\right) ,
\end{equation}
or
\begin{equation}\label{L8}
\mathcal{D}_{RL}^{\alpha }f\left( t\right) :=\left\{ 
\begin{array}{c}
\displaystyle\frac{1}{\Gamma \left( m-\alpha \right) }\frac{d^{m}}{dt^{m}}\left[
\int_{0}^{t}f\left( \tau \right) \left( t-\tau \right) ^{m-\alpha -1}d\tau %
\right] ,\text{ \ }m-1<\alpha <m ,\\ 
\\ 
\displaystyle\frac{d^{m}}{dt^{m}}f\left( t\right) \text{ },\text{\ \ \ \ \ \ \ \ \ \ \ \
\ \ \ \ \ \ \ \ \ \ \ \ \ \ \ \ \ \ \ \ \ \ \ \ \ \ \ \ \ \ }\alpha =m%
\end{array}%
\right.   
\end{equation}
where $D^{m} = d^m/dt^m$ is the usual derivative of integer-order $m$ and
$\textbf{J}^{m-\alpha }$ is the fractional integral in the Riemann-Liouville sense.
If  $\alpha=0$, we define $\mathcal{D}^{0}_{RL}=I$, where $I$ is the identity operator.
\end{definition}

\begin{definition} Let $\alpha\in\mathbb{C}$ such that $\mbox{Re}\left( \alpha
	\right) >0$ and $m$ the smallest integer greater than or equal to
	$\mbox{Re}\left( \alpha \right) >0$, with $m-1<\alpha \leq m$. The
	fractional derivative of order $\alpha$ of a causal function $f$ in the
	Caputo sense, ${D}_{C}^{\alpha }f\left( t\right)$, is defined by
	\rm\cite{IP,RM,RCEC}
\begin{equation}\label{L9}
\mathcal{D}_{C}^{\alpha }f\left( t\right) :=\textbf{J}^{m-\alpha }D^{m}f\left(
t\right) ,\text{ }m\in \mathbb{N}
\end{equation}
or 
\begin{equation}\label{L10}
\mathcal{D}_{C}^{\alpha }f\left( t\right) :=\left\{ 
\begin{array}{c}
\displaystyle\frac{1}{\Gamma \left( m-\alpha \right) }\int_{0}^{t}f^{\left( m\right)
}\left( \tau \right) \left( t-\tau \right) ^{m-\alpha -1}d\tau ,\text{ \ }%
m-1<\alpha <m ,\\ 
\\ 
\displaystyle\frac{d^{m}}{dt^{m}}f\left( t\right) \text{ },\text{ \ \ \ \ \ \ \ \ \ \ \ \ \ \ \ \ \ \ \ \ \ \ \ \ \ \ \ \ \ \ \ \ \ \ \ \ \ \ \ \ \ \ }\alpha =m%
\end{array}%
\right. 
\end{equation}
where $D^{m}$ is the usual derivative of integer-order $m$,
	$\textbf{J}^{m-\alpha }$ is the fractional integral in the
	Riemann-Liouville sense and $f^{(m)}(\tau) = \frac{d^m
	f(\tau)}{d\tau^m}$.
\end{definition}

We present now the definitions of two new types of $"fractional"$ derivatives.  As
we shall show later, these definitions coincide, for a particular value of
their parameters, with the derivative of order one of integer-order calculus.

\begin{definition} Let $f:\left[ 0,\infty \right) \rightarrow \mathbb{R}$ and $t>0$.
	Then the conformable fractional derivative of order $\alpha$ of 
	$f$ is defined by \rm\cite{RMAM}
\begin{equation}
T_{\alpha }f \left( t\right) =\underset{\varepsilon
\rightarrow 0}{\lim }\frac{f\left( t+\varepsilon t^{1-\alpha }\right)
-f\left( t\right) }{\varepsilon },
\end{equation}
$\forall t>0$ and $\alpha\in(0,1)$.
\end{definition}

A function $f$ is called $\alpha$-differentiable if it has a fractional derivative.
If $f$ is $\alpha$-differentiable in some interval $(0,a)$, $a>0$
and if $\underset{t\rightarrow 0^{+}}{\lim }f^{\left( \alpha \right)
}\left(t\right) $ exists, then we define

\begin{equation*}
T_{\alpha}\left( 0\right) =\underset{t\rightarrow 0^{+}}{%
\lim }T_{\alpha }f \left( t\right).
\end{equation*}

\begin{definition} Let $f:\left[ 0,\infty \right) \rightarrow \mathbb{R}$ and
	$t>0$.  Then the alternative fractional derivative of order $\alpha$ of
	$f$ is defined by \rm\cite{UNK}
\begin{equation}
\mathcal{D}_{\alpha }f\left( t\right) =\underset{\varepsilon
\rightarrow 0}{\lim }\frac{f\left( te^{\varepsilon t^{-\alpha}} \right)
-f\left( t\right) }{\varepsilon },
\end{equation}
$\forall t>0$ and $\alpha\in(0,1)$.
\end{definition}
Also, if $f$ is $\alpha$-differentiable in some interval $(0,a)$, $a>0$ and
$\underset{t\rightarrow 0^{+}}{\lim }f^{\left( \alpha \right) }\left(t\right) $
exists, then we define
\begin{equation*}
\mathcal{D}_{\alpha}\left( 0\right) =\underset{t\rightarrow 0^{+}}{%
\lim }\mathcal{D}_{\alpha }f\left( t\right) .
\end{equation*}

In this work, if both the conformable and the alternative fractional derivatives of order $\alpha$ of a function $f$ exist, we will simply said that the function $f$ is $\alpha$-differentiable.


Ortigueira and Machado \cite{JAM,MDJA} proposed that an operator can be considered a fractional derivative if it satisfies the following properties: (a) linearity; (b) identity; (c) compatibility with previous versions, that is, when the order is integer, the fractional derivative produces the same result as the ordinary integer-order derivative; (d) the law of exponents $\mathcal{D}^{\alpha}\mathcal{D}^{\beta}f(t)=\mathcal{D}^{\alpha+\beta}f(t)$, for all $\alpha<0$ and $\beta<0$; (e) generalized Leibniz rule $\mathcal{D}^{\alpha }\left( f\left( t\right) g\left( t\right) \right)
=\overset{\infty }{\underset{i=0}{\sum }}\binom{\alpha }{i}D^{i}f\left( t\right) D^{\alpha -i}g\left( t\right)$.\footnote{Note that when $\alpha=n\in \mathbb{Z}^{+}$ we obtain the classical Leibniz rule.}

Fractional derivatives in the Riemann-Liouville and in the Caputo sense satisfy such properties, but neither the conformable nor the alternative fractional derivatives satisfy them. 

On the other hand, it is possible to find undesirable characteristics even
in fractional derivatives that satisfy the criteria presented above, e.g.:

\begin{enumerate}

\item Most fractional derivatives do not satisfy $\mathcal{D}^{\alpha}(1)=0$ if
	$\alpha$ is not a natural number. An importante exception is the
		derivative in the Caputo sense.

\item Not all fractional derivatives obey the product rule for two
	functions: 
\begin{equation*}
\mathcal{D}^{\alpha }\left( f\cdot g\right) \left( t\right) =f\left(
t\right) D^{\alpha }g\left( t\right) +g\left( t\right) \mathcal{D}^{\alpha
}f\left( t\right) .
\end{equation*}

\item Not all fractional derivatives obey the quotient rule for two functions:
\begin{equation*}
D^{\alpha }\left( \frac{f}{g}\right) \left( t\right) =\frac{f\left( t\right)
D^{\alpha }g\left( t\right) -g\left( t\right) D^{\alpha }f\left( t\right) }{%
\left[ g\left( t\right) \right] ^{2}}
\end{equation*}

\item Not all fractional derivatives obey the chaim rule:
\begin{equation*}
\mathcal{D}^{\alpha }\left( f\circ g\right) \left( t\right) =g^{\left(
\alpha \right) }\left( t\right) f^{\left( \alpha \right) }\left( g\left(
t\right) \right) .
\end{equation*}

\item Fractional derivatives do not have a corresponding Rolle's theorem.

\item Fractional derivatives do not have a corresponding mean value theorem.

\item Fractional derivatives do not have a corresponding extended mean value theorem.

\item The definition of the Caputo derivative assumes that function $f$ is
	differentiable in the classical sense of the term. 

\end{enumerate}

In this sense, the new conformable and alternative fractional derivatives fit
perfectly into the classical properties of integer-order calculus, in
particular in the case of order one.

The objective of this work is to present a new type of derivative,
the local $M$-derivative, that generalizes the alternative fractional
derivative. The new definition seems to be a natural extension of the usual,
integer-order derivative, and satisfies the eight properties mentioned above.
Also, as in the case of  conformable and alternative fractional derivatives,
our definition coincides with the known fractional derivatives of polynomials.
Finally, we were able to define a corresponding integral for which we can prove
the fundamental theorem of calculus, the inversion theorem and a theorem of
integration by parts.


\section{Local $M$-derivative} 

In this section we present the main definition of this article and obtain
several results that generalize equivalent results valid for the alternative
fractional derivative and which bear a great similarity
to the results found in classical calculus.

On the basis of this definition we could demonstrate that our local $M$-derivative is linear, obeys the product rule, the composition rule for two $\alpha$-differentiable functions, the quotient rule and the chain rule. We show that the derivative of a constant is zero and present $"fractional"$ versions
of Rolle's theorem, the mean value theorem and the extended mean value theorem. Further, the continuity of the $"fractional"$ derivative is demonstrated, as in integer-order calculus. 

Thus, let us begin with the following definition, which is a generalization of
the usual definition of a derivative as a special limit. 

\begin{definition}\label{def5} Let $f:\left[ 0,\infty \right) \rightarrow \mathbb{R}$ and
	$t>0$. For $0<\alpha <1$ we define the local $M$-derivative of order
	$\alpha$ of function $f$, denoted $\mathscr{D}_{M}^{\alpha,\beta}
	f(t)$, by 
\begin{equation}\label{J}
\mathscr{D}_{M}^{\alpha,\beta }f\left( t\right) :=\underset{\varepsilon \rightarrow 0}{%
\lim }\frac {f\left( t\mathbb{E}_{\beta }\left( \varepsilon t^{-\alpha }\right)
\right) -f\left( t\right) }{\varepsilon },
\end{equation}
$\forall t>0$, where $\mathbb{E}_{\beta }\left(\cdot\right) $, $\beta >0$ is the
	Mittag-Leffler function with one parameter \rm\cite{GMM,GKAM}. Note that
	if $f$ is $\alpha$-differentiable in some interval
	$(0,a)$, $a>0$, and $\underset{t\rightarrow 0^{+}}{\lim
	}\mathscr{D}_{M}^{\alpha,\beta }f\left( t\right) $ exists, then we have

\begin{equation*}
\mathscr{D}_{M}^{\alpha,\beta }f\left( 0\right) =\underset{t\rightarrow 0^{+}}{\lim }%
\mathscr{D}_{M}^{\alpha,\beta }f\left( t\right) .
\end{equation*}
\end{definition}

\begin{theorem} If a function $f:\left[ 0,\infty \right) \rightarrow \mathbb{R}$ is $\alpha$-differentiable at $t_{0}>0$, $0<\alpha \leq 1$, $\beta>0$, then $f$ is continuous at $t_{0}$.
\end{theorem}

\begin{proof} Indeed, let us consider the identity
\begin{equation}\label{B}
f\left( t_{0}\mathbb{E}_{\beta }\left( \varepsilon t_{0}^{-\alpha }\right) \right)
-f\left( t_{0}\right) =\left( \frac{f\left( t_{0}\mathbb{E}_{\beta }\left(
\varepsilon t_{0}^{-\alpha }\right) \right) -f\left( t_{0}\right) }{%
\varepsilon }\right) \varepsilon .
\end{equation}
Applying the limit $\varepsilon \rightarrow 0$ on both sides of {\rm\text{Eq}.\rm(\ref{B})}, we have
\begin{eqnarray*}
\underset{\varepsilon \rightarrow 0}{\lim }f\left( t_{0}\mathbb{E}_{\beta
}\left( \varepsilon t_{0}^{-\alpha }\right) \right) -f\left( t_{0}\right) 
&=&\underset{\varepsilon \rightarrow 0}{\lim }\left( \frac{f\left(
t_{0}\mathbb{E}_{\beta }\left( \varepsilon t_{0}^{-\alpha }\right) \right) -f\left(
t_{0}\right) }{\varepsilon }\right) \underset{\varepsilon \rightarrow 0}{%
\lim }\varepsilon   \notag \\
&=&\mathscr{D}_{M}^{\alpha,\beta }f\left( t\right) \underset{\varepsilon \rightarrow 0}{\lim }\varepsilon   \notag \\
&=&0.
\end{eqnarray*}

Then, $f$ is continuous at $t_{0}$.
\end{proof}

Using the definition of the one-parameter Mittag-Leffler function, we have
\begin{equation}\label{A}
f\left( t\mathbb{E}_{\beta }\left( \varepsilon t^{-\alpha }\right) \right) =f\left( t%
\overset{\infty }{\underset{k=0}{\sum }}\frac{\left( \varepsilon t^{-\alpha
}\right) ^{k}}{\Gamma \left( \beta k+1\right) }\right) .
\end{equation}

Apply the limit $\varepsilon\rightarrow 0$ on both sides of Eq.(\ref{A}); 
since $f$ is a continuous function, we have
\begin{eqnarray}
\underset{\varepsilon \rightarrow 0}{\lim }f\left( t\mathbb{E}_{\beta }\left(
\varepsilon t^{-\alpha }\right) \right)  &=&\underset{\varepsilon
\rightarrow 0}{\lim }f\left( t\overset{\infty }{\underset{k=0}{\sum }}%
\frac{\left( \varepsilon t^{-\alpha }\right) ^{k}}{\Gamma \left( \beta
k+1\right) }\right)   \notag \\
&=&f\left( t\underset{\varepsilon \rightarrow 0}{\lim }\overset{\infty }{%
\underset{k=0}{\sum }}\frac{\left( \varepsilon t^{-\alpha }\right) ^{k}}{%
\Gamma \left( \beta k+1\right) }\right)   \notag \\
&=&f\left( t\right) ,
\end{eqnarray}
because when $\varepsilon \rightarrow 0$ in the Mittag-Leffler function,
the only term that contributes to the sum is $k=0$, so that 
\begin{equation*}
\underset{\varepsilon \rightarrow 0}{\lim }\overset{\infty }{\underset{%
k=0}{\sum }}\frac{\left( \varepsilon t^{-\alpha }\right) ^{k}}{\Gamma \left(
\beta k+1\right) }=1.
\end{equation*}

We present here a theorem that encompasses the main classical properties of
integer order derivatives, in particular of order one. As for the chain rule,
it will be verified by means of an example, as we will see in the sequence.

\begin{theorem}\label{A1} Let $0<\alpha \leq 1$, $\beta>0$, $a,b\in\mathbb{R}$ and $f, g$
	$\alpha$-differentiable at the point $t>0$. Then:

\begin{enumerate}

\item $ {\rm (Linearity)}\text{ } \mathscr{D}_{M}^{\alpha,\beta }\left( af+bg\right) \left( t\right) =a \mathscr{D}_{M}^{\alpha,\beta }f\left(
t\right) +b \mathscr{D}_{M}^{\alpha,\beta }g\left( t\right) $.

\begin{proof} Using {\rm\text{Definition 1}}, we have
\begin{eqnarray*}
\mathscr{D}_{M}^{\alpha,\beta }\left( af+bg\right) \left( t\right)  &=&\underset{\varepsilon
\rightarrow 0}{\lim }\frac{\left( af+bg\right) \left( t\mathbb{E}_{\beta }\left(
\varepsilon t^{-\alpha }\right) \right) -\left( af+bg\right) \left( t\right) 
}{\varepsilon }  \notag \\
&=&\underset{\varepsilon \rightarrow 0}{\lim }\frac{af\left( t\mathbb{E}_{\beta
}\left( \varepsilon t^{-\alpha }\right) \right) +bg\left( t\mathbb{E}_{\beta }\left(
\varepsilon t^{-\alpha }\right) \right) -af\left( t\right) -bg\left(
t\right) }{\varepsilon }  \notag \\
&=&\underset{\varepsilon \rightarrow 0}{\lim }\frac{af\left( t\mathbb{E}_{\beta
}\left( \varepsilon t^{-\alpha }\right) \right) -af\left( t\right) }{%
\varepsilon }+\underset{\varepsilon \rightarrow 0}{\lim }\frac{bg\left(
t\mathbb{E}_{\beta }\left( \varepsilon t^{-\alpha }\right) \right) -bg\left(
t\right) }{\varepsilon }  \notag \\
&=&a\mathscr{D}_{M}^{\alpha,\beta }f\left( t\right) +b\mathscr{D}_{M}^{\alpha,\beta }g\left( t\right) 
\end{eqnarray*}
\end{proof}

\item{\rm (Product Rule)}\text{ } $ \mathscr{D}_{M}^{\alpha,\beta }\left( f\cdot g\right) \left( t\right)
	=f\left( t\right) \mathscr{D}_{M}^{\alpha,\beta }g\left( t\right)
		+g\left( t\right) \mathscr{D}_{M}^{\alpha,\beta }f\left(
		t\right)$.

\begin{proof} Using {\rm Definition \ref{def5}}, we have

\begin{eqnarray*}
\mathscr{D}_{M}^{\alpha,\beta }\left( f\cdot g\right) \left( t\right)  &=&\underset%
{\varepsilon \rightarrow 0}{\lim }\frac{f\left( t\mathbb{E}_{\beta }\left(
\varepsilon t^{-\alpha }\right) \right) g\left( t\mathbb{E}_{\beta }\left(
\varepsilon t^{-\alpha }\right) \right) -f\left( t\right) g\left( t\right) }{%
\varepsilon } \\
&=&\underset{\varepsilon \rightarrow 0}{\lim }\left( \frac{%
\begin{array}{c}
f\left( t\mathbb{E}_{\beta }\left( \varepsilon t^{-\alpha }\right) \right)
g\left( t\mathbb{E}_{\beta }\left( \varepsilon t^{-\alpha }\right) \right)
+f\left( t\right) g\left( t\mathbb{E}_{\beta }\left( \varepsilon t^{-\alpha
}\right) \right) - \\ 
-f\left( t\right) g\left( t\mathbb{E}_{\beta }\left( \varepsilon t^{-\alpha
}\right) \right) -f\left( t\right) g\left( t\right) 
\end{array}%
}{\varepsilon }\right)  \\
&=&\underset{\varepsilon \rightarrow 0}{\lim }\left( \frac{f\left( t\mathbb{E%
}_{\beta }\left( \varepsilon t^{-\alpha }\right) \right) -f\left( t\right) }{%
\varepsilon }\right) \underset{\varepsilon \rightarrow 0}{\lim }g\left( t%
\mathbb{E}_{\beta }\left( \varepsilon t^{-\alpha }\right) \right) + \\
&&+\underset{\varepsilon \rightarrow 0}{\lim }\left( \frac{g\left( t\mathbb{E%
}_{\beta }\left( \varepsilon t^{-\alpha }\right) \right) -g\left( t\right) }{%
\varepsilon }\right) f\left( t\right)  \\
&=&\mathscr{D}_{M}^{\alpha,\beta }\left( f\right) \left( t\right) \underset{%
\varepsilon \rightarrow 0}{\lim }g\left( t\mathbb{E}_{\beta }\left(
\varepsilon t^{-\alpha }\right) \right) +\mathscr{D}_{M}^{\alpha,\beta }\left(
g\right) \left( t\right) f\left( t\right)  \\
&=&\mathscr{D}_{M}^{\alpha,\beta }\left( f\right) \left( t\right) g\left( t\right)
+\mathscr{D}_{M}^{\alpha,\beta }\left( g\right) \left( t\right) f\left( t\right) ,
\end{eqnarray*}
because $\underset{\varepsilon \rightarrow 0}{\lim }g\left( t\mathbb{E}_{\beta }\left(
\varepsilon t^{-\alpha }\right) \right) =g\left( t\right) $.
\end{proof}

\item{\rm (Quotient rule)}\text{ } $\displaystyle\mathscr{D}_{M}^{\alpha,\beta }\left( \frac{f}{g}\right)
	\left( t\right) =\frac{g\left(t\right) \mathscr{D}_{M}^{\alpha,\beta
	}f\left( t\right) -f\left( t\right) \mathscr{D}_{M}^{\alpha,\beta
	}g\left( t\right) }{\left[ g\left( t\right) \right] ^{2}} $.

\begin{proof} Using {\rm Definition \ref{def5}}, we have

\begin{eqnarray*}
\mathscr{D}_{M}^{\alpha,\beta }\left( \frac{f}{g}\right) \left( t\right)  &=&\underset{%
\varepsilon \rightarrow 0}{\lim }\frac{\displaystyle\frac{f\left( t\mathbb{E}_{\beta }\left(
\varepsilon t^{-\alpha }\right) \right) }{g\left( t\mathbb{E}_{\beta }\left(
\varepsilon t^{-\alpha }\right) \right) }-\frac{f\left( t\right) }{g\left(
t\right) }}{\varepsilon }  \notag \\
&=&\underset{\varepsilon \rightarrow 0}{\lim }\frac{\displaystyle\frac{g\left(
t\right) f\left( t\mathbb{E}_{\beta }\left( \varepsilon t^{-\alpha }\right) \right)
-f\left( t\right) g\left( t\mathbb{E}_{\beta}\left( \varepsilon t^{-\alpha }\right)
\right) +f\left( t\right) g\left( t\right) -f\left( t\right) g\left(
t\right) }{\varepsilon }}{g\left( t\mathbb{E}_{\beta }\left( \varepsilon t^{-\alpha
}\right) \right) g\left( t\right) }  \notag \\
&=&\frac{\underset{\varepsilon \rightarrow 0}{\lim }\displaystyle\frac{g\left(
t\right) \left( f\left( t\mathbb{E}_{\beta}\left( \varepsilon t^{-\alpha }\right)
\right) -f\left( t\right) \right) }{\varepsilon }-\underset{\varepsilon
\rightarrow 0}{\lim }\frac{f\left( t\right) \left( g\left( t\mathbb{E}_{\beta
}\left( \varepsilon t^{-\alpha }\right) \right) -g\left( t\right) \right) }{%
\varepsilon }}{\underset{\varepsilon \rightarrow 0}{\lim }g\left(
t\mathbb{E}_{\beta }\left( \varepsilon t^{-\alpha }\right) \right) g\left( t\right) }
\notag \\
&=&\frac{g\left( t\right) \mathscr{D}_{M}^{\alpha,\beta }f\left( t\right) -f\left( t\right)
\mathscr{D}_{M}^{\alpha,\beta }g\left( t\right) }{\left[ g\left( t\right) \right] ^{2}},
\end{eqnarray*}
because $\underset{\varepsilon \rightarrow 0}{\lim }g\left( t\mathbb{E}_{\beta }\left(
\varepsilon t^{-\alpha }\right) \right) =g\left( t\right) $.
\end{proof}

\item $ \mathscr{D}_{M}^{\alpha,\beta }\left( c\right) =0 $, where $f(t)=c$ is
	a constant. 

\begin{proof}
	The result follows directly from {\rm Definition \ref{def5}}.
\end{proof} 

\item If, furthermore, $f$ is differentiable, then
	$\mathscr{D}_{M}^{\alpha,\beta } \left( f\right)(t)=
	\displaystyle\frac{t^{1-\alpha}}{\Gamma \left(\beta +1\right) }
	\frac{df(t)}{dt}. $

\begin{proof} We can write 
\begin{eqnarray}
t\mathbb{E}_{\beta }\left( \varepsilon t^{-\alpha }\right)  &=&\overset{%
\infty }{\underset{k=0}{\sum }}\frac{\left( \varepsilon t^{-\alpha }\right)
^{k}}{\Gamma \left( \beta k+1\right) }=t+\frac{\varepsilon t^{1-\alpha }}{%
\Gamma \left( \beta +1\right) }+\frac{t\left( \varepsilon t^{-\alpha
}\right) ^{2}}{\Gamma \left( 2\beta +1\right) }+\frac{t\left( \varepsilon
t^{-\alpha }\right) ^{3}}{\Gamma \left( 3\beta +1\right) }+  \notag
\label{C} \\
&&+\cdot \cdot \cdot +\frac{t\left( \varepsilon t^{-\alpha }\right) ^{n}}{%
\Gamma \left( n\beta +1\right) }+\cdot \cdot \cdot   \notag \\
&=&t+\frac{\varepsilon t^{1-\alpha }}{\Gamma \left( \beta +1\right) }%
+O\left( \varepsilon ^{2}\right) .
\end{eqnarray}
Using {\rm Definition \ref{def5}} and introducing the change
\begin{equation*}
h=\varepsilon t^{1-\alpha }\left( \frac{1}{\Gamma \left(\beta +1\right) }+O\left( \varepsilon
\right) \right) \Rightarrow \varepsilon =\frac{h}{t^{1-\alpha }\left( \frac{1%
}{\Gamma \left(\beta +1\right) }+O\left( \varepsilon \right) \right) },
\end{equation*}
in {\rm\text{Eq}.\rm(\ref{C})}, we conclude that
\begin{eqnarray*}
\mathscr{D}_{M}^{\alpha,\beta }f\left( t\right)  &=&\underset{\varepsilon \rightarrow 0}{%
\lim }\frac{f\left( t+\displaystyle\frac{\varepsilon t^{1-\alpha }}{\Gamma \left(\beta +1\right) }+O\left(
\varepsilon ^{2}\right) \right) -f\left( t\right) }{\varepsilon }  \notag \\
&=&\underset{\varepsilon \rightarrow 0}{\lim }\frac{\displaystyle\frac{f\left(
t+h\right) -f\left( t\right) }{ht^{\alpha -1}}}{\frac{1}{\Gamma \left(\beta +1\right) }\left(
1+{\Gamma \left(\beta +1\right) } O\left( \varepsilon \right) \right) }  \notag \\
&=&\frac{t^{1- \alpha}}{\Gamma \left(\beta +1\right) }\underset{\varepsilon \rightarrow 0}{%
\lim }\frac{\displaystyle\frac{f\left( t+h\right) -f\left( t\right) }{h}}{1+{\Gamma \left(\beta +1\right) } O\left( \varepsilon \right) }  \notag \\
&=&\frac{t^{1- \alpha}}{\Gamma \left(\beta +1\right) }\frac{df\left( t\right) }{dt},
\end{eqnarray*}
with $\beta>0$ and $t>0$.
\end{proof}

\item{\rm\text{(Chain rule)}} $\mathscr{D}_{M}^{\alpha,\beta }\left(f\circ g\right)(t)=f'(g(t))
	\mathscr{D}_{M}^{\alpha,\beta }g(t)$, for $f$ differentiable at $g(t)$.

\begin{proof} If $g(t)=a$ is a constant, then
\begin{equation*}
\mathscr{D}_{M}^{\alpha,\beta }\left( f\circ g\right) \left( t\right) =\mathscr{D}_{M}^{\alpha,\beta }f\left( g\left( t\right) \right) =\mathscr{D}_{M}^{\alpha,\beta }\left( f\right) \left(
a\right) =0.
\end{equation*}

On the other hand, assume that $g$ is not a constant in the neighborhood of
	$a$, that is, suppose an $\varepsilon_{0}>0$ such that $g\left(
	x_{1}\right) \neq g\left(x_{2}\right) $, $\forall x_{1},x_{2}\in \left(
	a_{0}-\varepsilon _{0},a_{0}+\varepsilon_{0}\right) $. Now, since $g$
	is continuous at $a$, for $\varepsilon $ small enough we have

\begin{eqnarray}
\mathscr{D}_{M}^{\alpha ,\beta }\left( f\circ g\right) \left( a\right)  &=&%
\underset{\varepsilon \rightarrow 0}{\lim }\frac{f\left( g\left( a\mathbb{E}%
_{\beta }\left( \varepsilon t^{-\alpha }\right) \right) \right) -f\left(
g\left( a\right) \right) }{\varepsilon }  \label{D} \\
&=&\underset{\varepsilon \rightarrow 0}{\lim }\frac{f\left( g\left( a\mathbb{%
E}_{\beta }\left( \varepsilon t^{-\alpha }\right) \right) \right) -f\left(
g\left( a\right) \right) }{g\left( a\mathbb{E}_{\beta }\left( \varepsilon
t^{-\alpha }\right) \right) -g\left( a\right) }\frac{g\left( a\mathbb{E}%
_{\beta }\left( \varepsilon t^{-\alpha }\right) \right) -g\left( a\right) }{%
\varepsilon }.  \notag
\end{eqnarray}

Introducing the change
\begin{equation*}
\varepsilon _{1}=g\left( a\mathbb{E}_{\beta }\left( \varepsilon t^{-\alpha }\right)
\right) -g\left( a\right) \Rightarrow g\left( a\mathbb{E}_{\beta }\left( \varepsilon
t^{-\alpha }\right) \right) =g\left( a\right) +\varepsilon _{1}
\end{equation*}
in {\rm\text{Eq}.\rm(\ref{D})}, we conclude that
\begin{eqnarray*}
\mathscr{D}_{M}^{\alpha,\beta }\left( f\circ g\right) \left( a\right)  &=&\underset{%
\varepsilon _{1}\rightarrow 0}{\lim }\frac{f\left( g\left( a\mathbb{E}_{\beta
}\left( \varepsilon t^{-\alpha }\right) \right) \right) -f\left( g\left(
a\right) \right) }{\varepsilon _{1}}\underset{\varepsilon \rightarrow 0}{%
\lim }\frac{g\left( a\mathbb{E}_{\beta }\left( \varepsilon t^{-\alpha }\right)
\right) -g\left( a\right) }{\varepsilon }  \notag \\
&=&f^{\prime }\left( g\left( a\right) \right) \mathscr{D}_{M}^{\alpha,\beta }g\left(
a\right) ,
\end{eqnarray*}
with $a>0$.
\end{proof}
\end{enumerate}
\end{theorem}

The two theorems below, present some results used in the calculus of integer order.

\begin{theorem}\label{A2} Let $a \in\mathbb{R}$, $\beta>0$ and $0<\alpha \leq
	1$. Then we have the following results:
	
\begin{enumerate}
\item $\mathscr{D}_{M}^{\alpha,\beta }\left( 1\right) =0$.
\item $\mathscr{D}_{M}^{\alpha,\beta }\left( e^{at}\right) =\displaystyle\frac{t^{1-\alpha }}{\Gamma \left(\beta +1\right) }ae^{at}$.
\item $\mathscr{D}_{M}^{\alpha,\beta }\left( \sin \left( at\right) \right) =\displaystyle\frac{t^{1-\alpha }}{\Gamma \left(\beta +1\right) }a\cos \left( at\right)$ .
\item $\mathscr{D}_{M}^{\alpha,\beta }\left( \cos \left( at\right) \right) =-\displaystyle\frac{t^{1-\alpha }}{\Gamma \left(\beta +1\right) }a\sin \left( at\right) $.
\item $\mathscr{D}_{M}^{\alpha,\beta }\left( \frac{t^{\alpha }}{\alpha }\right) =\displaystyle\frac{1}{\Gamma \left(\beta +1\right) }$.
\item $\mathscr{D}_{M}^{\alpha,\beta }\left( t^{a}\right) =\displaystyle\frac{a}{\Gamma \left(\beta +1\right) }t^{a-\alpha}$.
\end{enumerate}
\end{theorem}

\begin{theorem}\label{A3} Let $0<\alpha \leq 1$, $\beta>0$ and $t>0$. Then we
	have the following results:

\begin{enumerate}

\item $\mathscr{D}_{M}^{\alpha,\beta }\left( \sin \left( \frac{1}{\alpha }t^{\alpha }\right)
\right) =\displaystyle\frac{\cos \left( \frac{1}{\alpha }t^{\alpha }\right) }{\Gamma \left(\beta +1\right) }.$

\item $\mathscr{D}_{M}^{\alpha,\beta }\left( \cos \left( \frac{1}{\alpha }t^{\alpha }\right)
\right) =-\displaystyle\frac{\sin \left( \frac{1}{\alpha }t^{\alpha }\right) }{\Gamma \left(\beta +1\right) }.$

\item $\mathscr{D}_{M}^{\alpha,\beta }\left( e^{\frac{t^{\alpha }}{\alpha }}\right) =\displaystyle\frac{e^{\frac{%
t^{\alpha }}{\alpha }}}{\Gamma \left(\beta +1\right) }.$

\end{enumerate}

\end{theorem}

The identities in Theorem \ref{A2} and Theorem \ref{A3} are direct consequences of item 5 of Theorem \ref{A1}.

We now prove the extensions, for $\alpha$-differentiable functions in the sense of the local $M$-derivative defined in Eq.(\ref{J}), of  Rolle's theorem and the mean value and extended mean value theorems. 

\begin{theorem} {\rm\text{(Rolle's theorem for $\alpha$-differentiable functions)}}
Let $a>0$ and $f:\left[ a,b\right] \rightarrow \mathbb{R}$ be a function such that:  
\begin{enumerate}
\item $f$ is continuous on $[a,b]$;
\item $f$ is $\alpha$-differentiable on $(a,b)$ for some $\alpha\in(0,1)$;
\item  $f(a)=f(b)$.
\end{enumerate}
Then, there exists $c\in(a,b)$, such that $\mathscr{D}^{\alpha,\beta}_{M}f(c)=0$, $\beta>0$.
\end{theorem}

\begin{proof} Since $f$ is continuous on
	$[a,b]$ and $f(a)=f(b)$, there exists a point $c\in(a,b)$ at which
	function $f$ has a local extreme. Then,
\begin{eqnarray*}
\mathscr{D}_{M}^{\alpha,\beta }f\left( c\right)  &=&\underset{\varepsilon \rightarrow 0^{-}}{%
\lim }\frac{f\left( c\mathbb{E}_{\beta }\left( \varepsilon c^{-\alpha }\right)
\right) -f\left( c\right) }{\varepsilon }=\underset{\varepsilon \rightarrow 0^{+}}{\lim }\frac{f\left( c \mathbb{E}_{\beta
}\left( \varepsilon c^{-\alpha }\right) \right) -f\left( c\right) }{%
\varepsilon }.
\end{eqnarray*}
As $\underset{\varepsilon \rightarrow 0^{\pm}}{\lim} 
\mathbb{E}_{\beta }\left(\varepsilon c^{-\alpha }\right) =
	1$,  the two limits in the right hand side of this equation have opposite signs. Hence,
	$\mathscr{D}_{M}^{\alpha,\beta }f\left( c\right) =0$.  
\end{proof}

\begin{theorem} {\rm\text{(Mean value theorem for $\alpha$-differentiable functions)}}
Let $a>0$ and $f:\left[ a,b\right] \rightarrow \mathbb{R}$ be a function such that:
\begin{enumerate}
\item $f$ is continuous on $[a,b]$;
\item $f$ is $\alpha$-differentiable on $(a,b)$ for some $\alpha\in(0,1)$.
\end{enumerate}
Then, there exists $c\in(a,b)$ such that
\begin{equation*}
\mathscr{D}_{M}^{\alpha,\beta }f\left( c\right) =\frac{f\left( b\right) -f\left(
a\right) }{\displaystyle\frac{b^{\alpha }}{\alpha }-\frac{a^{\alpha }}{\alpha }},
\end{equation*}
with $\beta>0$.
\end{theorem}
\begin{theorem}
Consider the function 
\begin{equation}\label{E}
g\left( x\right) =f\left( x\right) -f\left( a\right) -{\Gamma \left(\beta +1\right) }\left( \displaystyle\frac{%
f\left( b\right) -f\left( a\right) }{\displaystyle\frac{1}{\alpha }b^{\alpha }-\displaystyle\frac{1}{%
\alpha }a^{\alpha }}\right) \left( \frac{1}{\alpha }x^{\alpha }-\frac{1}{
\alpha }a^{\alpha }\right) .
\end{equation}

Function $g$ satisfies the conditions of Rolle's theorem. Then, there exists 
	$c\in(a,b)$ such that $\mathscr{D}_{M}^{\alpha,\beta }f(c)=0$. 
	Applying the local $M$-derivative $\mathscr{D}_{M}^{\alpha,\beta }$ to both sides
	of {\rm\text{Eq}.\rm(\ref{E})} and using the fact that
	$\mathscr{D}_{M}^{\alpha,\beta }\left( \frac{t^{\alpha }}{\alpha
	}\right) =\frac{1}{\Gamma \left(\beta +1\right) }$ and 
	$\mathscr{D}_{M}^{\alpha,\beta }(c) = 0$, with $c$ a constant, we
	conclude that

\begin{equation*}
\mathscr{D}_{M}^{\alpha,\beta }f\left( c\right) =\frac{f\left( b\right) -f\left(
a\right) }{\displaystyle\frac{b^{\alpha }}{\alpha }-\frac{a^{\alpha }}{\alpha }}.
\end{equation*}
\end{theorem}

\begin{theorem}{\rm\text{(Extended mean value theorem for fractional 	$\alpha$-differentiable functions)}} Let $a>0$ and $f,g:\left[a,b\right] \rightarrow \mathbb{R}$ functions such that:
\begin{enumerate}
\item $f, g$ are continuous on $[a,b]$;
\item $f, g$ are $\alpha$-differentiable for some $\alpha\in(0,1)$.
\end{enumerate}
Then, there exists $c\in(a,b)$ such that
\begin{equation}
\frac{\mathscr{D}_{M}^{\alpha,\beta }f\left( c\right) }{\mathscr{D}_{M}^{\alpha,\beta }g\left( c\right) }=\frac{f\left( b\right) -f\left( a\right) }{g\left(
b\right) -g\left( a\right) },
\end{equation}
with $\beta>0$.
\end{theorem}

\begin{proof}
Consider the function
\begin{equation}\label{F}
F\left( x\right) =f\left( x\right) -f\left( a\right) -\left( \frac{f\left(
b\right) -f\left( a\right) }{g\left( b\right) -g\left( a\right) }\right)
\left( g\left( x\right) -g\left( a\right) \right). 
\end{equation}

As $F$ is continuous on $[a,b]$, $\alpha$-differentiable on $(a,b)$ and
	$F(a)=0=F(b)$, by Rolle's theorem there exists $c\in(a,b)$ such that
	$\mathscr{D}_{M}^{\alpha,\beta }F(c)=0$ for some $\alpha\in(0,1)$.
	Then, applying the local $M$-derivative
	$\mathscr{D}_{M}^{\alpha,\beta }$ to both sides of
	{\rm\text{Eq}.\rm(\ref{F})} and using the fact that
	$\mathscr{D}_{M}^{\alpha,\beta }(c) = 0$ when $c$ is a constant, we
	conclude that
\begin{equation*}
\frac{\mathscr{D}_{M}^{\alpha,\beta }f\left( c\right) }{\mathscr{D}_{M}^{\alpha,\beta }g\left( c\right) }=\frac{f\left( b\right) -f\left( a\right) }{g\left(
b\right) -g\left( a\right) }.
\end{equation*}
\end{proof}

\begin{definition} Let $\beta>0$, $\alpha\in(n,n+1]$, for some $n\in\mathbb{N}$ and $f$ 	$n$ times differentiable \rm{(in the classical sense)} for $t>0$. Then the local $M$-derivative of order $n$ of $f$ is defined by 
	\begin{equation}
		\mathscr{D}_{M}^{\alpha,\beta;n}f\left( t\right)
		:=\underset{\varepsilon \rightarrow 0}{\lim }\frac
			{f^{(n)}\left( t\mathbb{E}_{\beta }\left( \varepsilon
			t^{n-\alpha }\right) \right) -f^{(n)}\left( t\right)
			}{\varepsilon },
	\end{equation} 
	if and only if the limit exists.  \end{definition}

From Definition 6 and the chain rule, that is from item 5 of Theorem 2, by
induction on $n$, we can prove that $\mathscr{D}_{M}^{\alpha,\beta,n}f\left(
t\right)= \displaystyle\frac{t^{n+1-\alpha}}{\Gamma \left(\beta +1\right) }
f^{(n+1)}(t)$, $\alpha\in(n,n+1]$ and so $f$ is
$(n+1)$-differentiable for $t>0$.

\section{$M$-integral}

In this section we introduce the concept of $M$-integral of a function $f$. From this definition we can prove some results similar to classical results such as the inverse property, the fundamental theorem of calculus and the theorem of integration by parts. Other results about the $M$-integral are also presented. In preparing this section we made extensive use of references \cite{RMAM,UNK,OSEN,TA}.

\begin{definition}{\rm\text{($M$-integral)}} Let $a\geq 0$ and $t\geq
	a$. Let $f$ be a function defined in $(a,t]$ and $0<\alpha <1$.
	Then, the $M$-integral of order $\alpha$ of a function $f$
	is defined by
\begin{equation}\label{ZA}
_{M}\mathcal{I}_{a}^{\alpha,\beta }f\left( t\right) ={\Gamma \left(\beta
	+1\right) }\int_{a}^{t}\frac{f\left( x\right) }{ x^{1-\alpha }}dx,
\end{equation}
with $\beta>0$.

\end{definition}

\begin{theorem} {\rm\text{(Inverse)}} 
	Let $a\geq 0$ and $0<\alpha < 1$. Also, let $f$ be a continuous function such that
	there exists $_{M}\mathcal{I}_{a}^{\alpha,\beta }f$. Then 
\begin{equation}
\mathscr{D}_{M}^{\alpha,\beta }(_{M}\mathcal{I}_{a}^{\alpha,\beta }f\left( t\right)) =f(t), 
\end{equation}
with $t\geq a$ and $\beta>0$.
\end{theorem}
 
\begin{proof} Indeed, using the chain rule proved in {\rm\text{Theorem} \rm\ref{A1}} we have
\begin{eqnarray}
\mathscr{D}_{M}^{\alpha,\beta }\left( _{M}\mathcal{I}_{a}^{\alpha,\beta }f\left( t\right) \right)  &=&\frac{%
t^{1-\alpha }}{\Gamma \left(\beta +1\right) }\frac{d}{dt}(_{M}\mathcal{I}_{a}^{\alpha,\beta }f\left( t\right))  \notag
\\
&=&\frac{t^{1-\alpha }}{\Gamma \left(\beta +1\right) }\frac{d}{dt}\left( {\Gamma \left(\beta +1\right) }\int_{a}^{t}\frac{%
f\left( x\right) }{x^{1-\alpha }}dx\right)   \notag \\
&=&\frac{t^{1-\alpha }}{\Gamma \left(\beta +1\right) }\left( \frac{\Gamma \left(\beta +1\right) }{t^{1-\alpha }}f\left(
t\right) \right)   \notag \\
&=&f\left( t\right) .
\end{eqnarray}
\end{proof}

We now prove the fundamental theorem of calculus in the sense of the $M$-derivative mentioned at the beginning of the paper.

\begin{theorem}\label{JK} {\rm\text{(Fundamental theorem of calculus)}} Let
	$f:\left( a,b\right) \rightarrow \mathbb{R}$ be an
	$\alpha$-differentiable function and $0<\alpha
	\leq 1$. Then, for all $t>a$ we have

\begin{equation*}
_{M}\mathcal{I}_{a}^{\alpha,\beta }\left(\mathscr{D}_{M}^{\alpha,\beta
}f\left( t\right) \right) =f\left( t\right) -f\left( a\right),
\end{equation*}
with $\beta>0$.
\end{theorem}

\begin{proof} In fact, since function $f$ is differentiable, using the
	chain rule of {\rm Theorem} {\rm\ref{A1}} and the  
	fundamental theorem of calculus for the integer-order derivative, we have
\begin{eqnarray}\label{Z1}
_{M}\mathcal{I}_{a}^{\alpha,\beta }\left( \mathscr{D}_{M}^{\alpha,\beta }f\left( t\right) \right)  &=&{\Gamma \left(\beta +1\right) }\int_{a}^{t}\frac{\mathscr{D}_{M}^{\alpha,\beta }f\left( t\right) }{x^{1-\alpha }}dx  \notag
\\
&=&{\Gamma \left(\beta +1\right) }\int_{a}^{t}\frac{x^{1-\alpha }}{\Gamma \left(\beta +1\right) }\frac{1}{x^{1-\alpha }}%
\frac{df\left( t\right) }{dt}dx  \notag \\
&=&\int_{a}^{t}\frac{df\left( t\right) }{dt}dx  \notag \\
&=&f\left( t\right) -f\left( a\right) 
\end{eqnarray}
\end{proof}

If the condition $f(a)=0$ holds, then by Theorem 9, Eq.(\ref{Z1}), we have $_{M}\mathcal{I}_{a}^{\alpha,\beta } \left[ {\mathscr{D}^{\alpha\,\beta}_{M}}f(t)\right]=f(t)$.

Theorem \ref{JK} can be generalized to a larger order as follows.

\begin{theorem} Let $\alpha\in(n,n+1]$ and $f:\left[a,\infty\right) \rightarrow
	\mathbb{R}$ be $(n+1)$-times differentiable for $t>a$. Then, $\forall$
	$t>a$, we have
\begin{equation*}
_{M}\mathcal{I}_{a}^{\alpha,\beta }\left( \mathscr{D}_{M}^{\alpha,\beta }f\left( t\right)
\right) =f\left( t\right) -\underset{k=0}{\overset{n}{\sum }}\frac{f^{\left(
k\right) }\left( a\right) \left( t-a\right) ^{k}}{k!},
\end{equation*}
with $\beta>0$.
\end{theorem}

\begin{proof} Using the definition of $M$-integral and the chain rule proved
	in {\rm\text{Theorem} \rm\ref{A1}}, we have
\begin{eqnarray}\label{L1}
_{M}\mathcal{I}_{a}^{\alpha,\beta }\left( \mathscr{D}_{M}^{\alpha,\beta }f\left( t\right)
\right)  &=&_{M}\mathcal{I}_{a}^{n+1}\left[ \left( t-a\right) ^{\beta -1}\mathscr{D}_{M}^{\alpha,\beta }f^{\left( n\right) }\left( t\right) \right]   \notag \\
&=&_{M}\mathcal{I}_{a}^{n+1}\left[ \left( t-a\right) ^{\beta -1}\left( t-a\right)
^{1-\beta }f^{\left( n+1\right) }\left( t\right) \right]   \notag \\
&=&_{M}\mathcal{I}_{a}^{n+1}\left( f^{\left( n+1\right) }\left( t\right) \right) .
\end{eqnarray}

Then, performing piecewise integration for the integer-order derivative in
	{\rm\text{Eq}.\rm(\ref{L1})}, we have
\begin{equation*}
_{M}\mathcal{I}_{a}^{\alpha,\beta }\left( \mathscr{D}_{M}^{\alpha,\beta }f\left( t\right)
\right) =f\left( t\right) -\underset{k=0}{\overset{n}{\sum }}\frac{f^{\left(
k\right) }\left( a\right) \left( t-a\right) ^{k}}{k!}.
\end{equation*}
\end{proof}

As integer-order calculus has a result known as integration by parts, we shall
now present, through a theorem, a similar result which we might call fractional integration by parts. 

We shall use the notation of Eq.(\ref{ZA}) for the $M$-integral:  

\begin{equation*}
_{M}\mathcal{I}_{a}^{\alpha,\beta }f\left( t\right) ={\Gamma \left(\beta +1\right) }\int_{a}^{t}\frac{f\left( x\right) }{%
x^{1-\alpha }}dx=\int_{a}^{t}f\left( x\right) d_{\alpha }x,
\end{equation*}
where $d_{\alpha }x=\dfrac{\Gamma \left(\beta +1\right) }{x^{1-\alpha }}dx$.

\begin{theorem} Let $f,g:[a,b]\rightarrow\mathbb{R}$ be two functions such
	that $f,g$ are differentiable and $0<\alpha<1$. Then
\begin{equation}
\int_{a}^{b}f\left( x\right) \mathscr{D}_{M}^{\alpha,\beta }g\left( x\right) d_{\alpha
}x=f\left( x\right) g\left( x\right) \mid _{a}^{b}-\int_{a}^{b}g\left(
x\right) \mathscr{D}_{M}^{\alpha,\beta }f\left( x\right) d_{\alpha }x,
\end{equation}
with $\beta>0$.
\end{theorem}

\begin{proof}Indeed, using the definition of $M$-integral and
	applying the chain rule of {\rm\text{Theorem} \rm\ref{A1}} and the
	fundamental theorem of calculus for integer-order derivatives, we have
\begin{eqnarray*}
\int_{a}^{b}f\left( x\right)\mathscr{D}_{M}^{\alpha,\beta }g\left( x\right) d_{\alpha }x
&=&{\Gamma \left(\beta +1\right) }\int_{a}^{b}\frac{f\left( x\right) }{x^{1-\alpha }}\mathscr{D}_{M}^{\alpha,\beta }g\left( x\right) dx  \notag \\
&=&{\Gamma \left(\beta +1\right) }\int_{a}^{b}\frac{f\left( x\right) }{x^{1-\alpha }}\frac{%
x^{1-\alpha }}{\Gamma \left(\beta +1\right) }\frac{dg\left( x\right) }{dt}dx  \notag \\
&=&\int_{a}^{b}f\left( x\right) g^{\prime }\left( x\right) dx  \notag \\
&=&f\left( x\right) g\left( x\right) \mid _{a}^{b}-{\Gamma \left(\beta +1\right) }\int_{a}^{b}\frac{%
g\left( x\right) }{x^{1-\alpha }}\frac{x^{1-\alpha }}{\Gamma \left(\beta +1\right) }\frac{df\left(
x\right) }{dt}dx  \notag \\
&=&f\left( x\right) g\left( x\right) \mid _{a}^{b}-\int_{a}^{b}g\left(
x\right) \mathscr{D}_{M}^{\alpha,\beta }f\left( x\right) d_{\alpha }x,
\end{eqnarray*}
where $d_{\alpha }x=\frac{\Gamma \left(\beta +1\right) }{x^{1-\alpha }}dx$.
\end{proof}

\begin{theorem}\label{A22} Let $0<a<b$ and let $f:\left[ a,b \right] \rightarrow
	\mathbb{R}$ be a continuous function. Then, for $0<\alpha <1$ we have

\begin{equation}
\left\vert _{M}\mathcal{I}_{a}^{\alpha,\beta }f\left( t\right) \right\vert \leq (_{M}\mathcal{I}_{a}^{\alpha,\beta }\left\vert f\left( t\right) \right\vert),
\end{equation}
with $\beta>0$.
\end{theorem}
\begin{proof} From the definition of $M$-integral of order $\alpha$, we have 
\begin{eqnarray*}
\left\vert _{M}\mathcal{I}_{a}^{\alpha,\beta }f\left( t\right) \right\vert  &=&\left\vert {\Gamma \left(\beta +1\right) }\int_{a}^{t}\frac{f\left( x\right) }{x^{1-\alpha }}dx\right\vert   \notag \\
&\leq &\left\vert {\Gamma \left(\beta +1\right) }\right\vert \int_{a}^{t}\left\vert \frac{f\left(
x\right) }{x^{1-\alpha }}\right\vert dx  \notag \\
&=&_{M}\mathcal{I}_{a}^{\alpha,\beta }\left\vert f\left( t\right) \right\vert .
\end{eqnarray*}
\end{proof} 

\begin{corollary} Let $f:\left[ a,b \right] \rightarrow \mathbb{R}$ be a continuous function such that
\begin{equation}
N=\underset{t\in \left[ a,b\right] }{\sup }\left\vert f\left( t\right)
\right\vert .
\end{equation}
Then, $\forall t\in \left[ a,b\right] $ and $0<\alpha <1$, we have
\begin{equation}
\left\vert _{M}\mathcal{I}_{a}^{\alpha,\beta }f\left( t\right) \right\vert \leq {\Gamma \left(\beta +1\right) } N\left( 
\frac{t^{\alpha }}{\alpha }-\frac{a^{\alpha }}{\alpha }\right),
\end{equation}
with $\beta>0$.
\end{corollary}
\begin{proof} By {\rm Theorem \rm\ref{A22}}, we have
\begin{eqnarray}
\left\vert _{M}\mathcal{I}_{a}^{\alpha,\beta }f\left( t\right) \right\vert  &\leq &_{M}\mathcal{I}_{a}^{\alpha,\beta }\left\vert f\left( t\right) \right\vert   \notag \\
&=&{\Gamma \left(\beta +1\right) }\int_{a}^{t}\left\vert f\left( x\right) \right\vert x^{\alpha -1}dx
\notag \\
&\leq &{\Gamma \left(\beta +1\right) }N\int_{a}^{t}x^{\alpha -1}dx  \notag \\
&=&{\Gamma \left(\beta +1\right) }N\left( \frac{t^{\alpha }}{\alpha }-\frac{a^{\alpha }}{\alpha }%
\right) .
\end{eqnarray}
\end{proof}

\section{Relation with alternative fractional derivative}

In this section we discuss the relation between the alternative fractional derivative and the local $M$-derivative proposed here.

Katugampola \cite{UNK} proposed a new fractional derivative which he called
alternative fractional derivative, given by
\begin{equation}\label{K}
\mathcal{D}^{\alpha }f\left( t\right) =\underset{\varepsilon \rightarrow 0}{\lim }%
\frac{f\left( te^{\varepsilon t^{-\alpha }}\right) -f\left( t\right) }{%
\varepsilon },
\end{equation}
with $\alpha\in(0,1)$ and $t>0$.

It is easily seen that our definition of local $M$-derivative Eq.(\ref{J})) is more general than the alternative fractional derivative Eq.(\ref{K}).

The definition in Eq.(\ref{J}) contains the one parameter Mittag-Leffler
function $\mathbb{E}_{\beta}(\cdot)$, which can be considered a generalization
of the exponential function. Indeed, choosing $\beta=1$ in the definition of
the one parameter Mittag-Leffler function \cite{GMM,GKAM}, we have

\begin{equation}
\mathbb{E}_{\beta }\left( x\right) = \mathbb{E}_{1}\left( x\right) =\overset{\infty }{\underset{%
k=0}{\sum }}\frac{x^{k}}{\Gamma \left( k+1\right) }=e^{x}.
\end{equation}
In particular, introducing $x=\varepsilon t^{-\alpha}$ in Eq.(\ref{J}) and taking the limit $\varepsilon \rightarrow 0$ we recover the 
alternative fractional derivative $\mathcal{D}^{\alpha}$: 
\begin{equation}
\mathscr{D}_{M}^{\alpha,\beta }f\left( t\right) =\underset{\varepsilon \rightarrow 0}{\lim }%
\frac{f\left( t\mathbb{E}_{1}\left( \varepsilon t^{-\alpha }\right) \right) -f\left(
t\right) }{\varepsilon }=\underset{\varepsilon \rightarrow 0}{\lim }\frac{%
f\left( te^{\varepsilon t^{-\alpha }}\right) -f\left( t\right) }{\varepsilon 
}=\mathcal{D}^{\alpha }f\left( t\right) . 
\end{equation}

\section{Application}

Fractional linear differential equations are important in the study of fractional calculus and applications. In this section, we present the general solution of a linear differential equation by means of the local $M$-derivative. In this sense, as a particular case, we study an example and perform a graph analysis of the solution.

The general first order differential equation based on the local $M$-derivative is represented by
\begin{equation}\label{K2}
\mathscr{D}_{M}^{\alpha ,\beta }u(t)+P\left( t\right) u\left( t\right) =Q\left(
t\right).
\end{equation}
where $P\left( t\right) ,$ $Q\left( t\right) $ are $\alpha-$differentiable functions and $u\left( t\right) $ is unknown.

Using the item 5 of Theorem (\ref{A1}) in the Eq.(\ref{K2}), we have
\begin{equation}\label{K3}
\frac{d}{dt}u\left( t\right) +\frac{\Gamma \left( \beta +1\right) }{%
t^{1-\alpha }}P\left( t\right) u\left( t\right) =\frac{\Gamma \left( \beta
+1\right) }{t^{1-\alpha }}Q\left( t\right).
\end{equation}

The Eq.(\ref{K3}) is a first order equation, whose general solution is given by
\begin{equation*}
u\left( t\right) =e^{-\Gamma \left( \beta +1\right) \int \frac{P\left(
t\right) }{t^{1-\alpha }}dt}\left( \Gamma \left( \beta +1\right) \int \frac{%
Q\left( t\right) }{t^{1-\alpha }}e^{\Gamma \left( \beta +1\right) \int \frac{%
P\left( t\right) }{t^{1-\alpha }}dt}dt+C\right),
\end{equation*}
where $C$ is an arbitrary constant.

By definition of $M$-integral, we conclude that the solution is given by
\begin{equation*}
u\left( t\right) =e^{-\text{ }_{M}I_{a}^{\alpha ,\beta }\left( P\left(
t\right) \right) }\left( _{M}I_{a}^{\alpha ,\beta }\left( Q\left( t\right)
e^{_{M}I_{a}^{\alpha ,\beta }\left( P\left( t\right) \right) }\right)
+C\right).
\end{equation*}

Now let us choose some values and functions and make an example using the linear differential equation previously studied by means of the local $M$-derivative. Then, taking $P(t)=-\lambda$, $Q(t)=0$, $u(0)=u_{0}$, $a=0$, $0<\alpha\leq 1$ e $\beta>0$, we have the following linear differential equation \begin{equation}\label{K5}
\mathscr{D}_{M}^{\alpha ,\beta }u(t)=\lambda u\left( t\right),
\end{equation}
whose solution is given by
\begin{equation*}
u\left( t\right) = u_{0} \text{ }e^{\frac{-\lambda }{\alpha }\Gamma \left( \beta
+1\right) t^{\alpha }}=u_{0}\mathbb{E}_{1}\left( {\frac{-\lambda }{\alpha }\Gamma \left( \beta
+1\right) t^{\alpha }}\right),
\end{equation*}
where $\mathbb{E}_{1}(\cdot)$ is Mittag-Leffler function.

\begin{figure}[h!]
\caption{Analytical solution of the {\rm Eq.(\ref{K5})}. We consider the values $ \beta = 0.5 $, $\lambda$=1 and $u_{0}$=20.}
\centering 
\includegraphics[width=12cm]{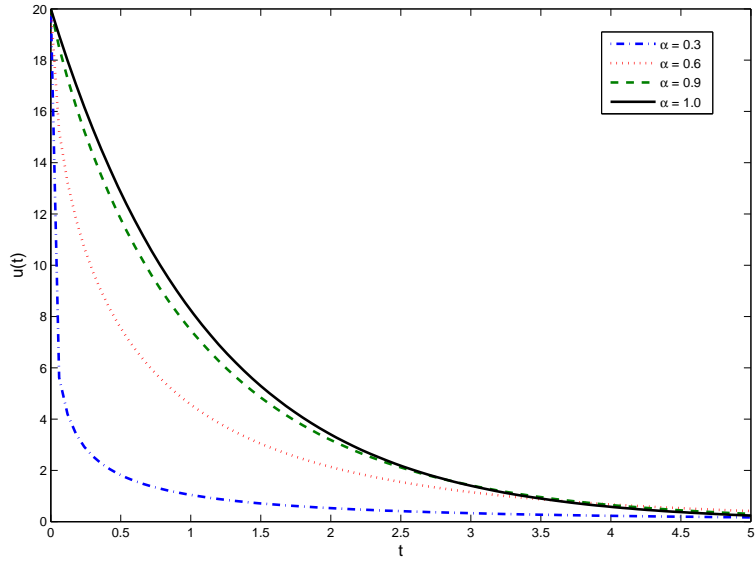} 
\label{fig:eryt1}
\end{figure}

\begin{figure}[h!]
\caption{Analytical solution of the {\rm Eq.(\ref{K5})}. We take the values $ \beta = 1.0 $, $\lambda$=2 and $u_{0}$=20}
\centering 
\includegraphics[width=12cm]{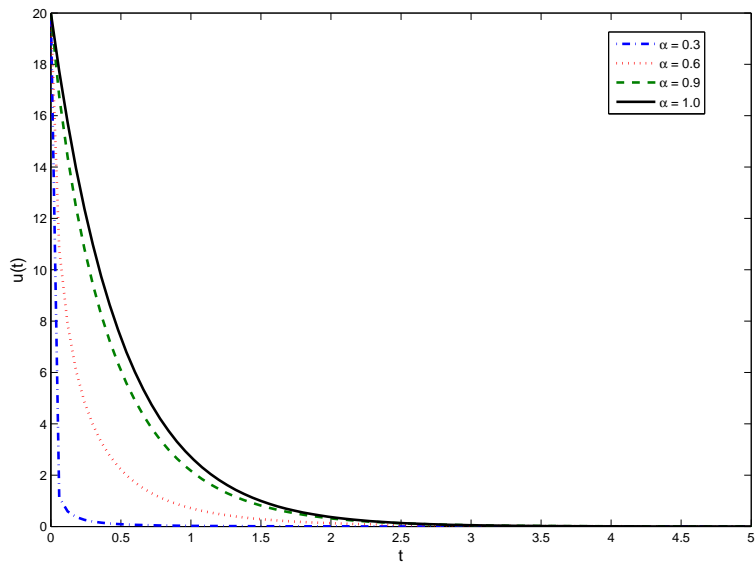} 
\label{fig:eryt1}
\end{figure}
\newpage
\begin{figure}[h!]
\caption{Analytical solution of the {\rm Eq.(\ref{K5})}. We chose the values $ \beta = 1.5 $, $\lambda$=2.5 and $u_{0}$=20.}
\centering 
\includegraphics[width=12cm]{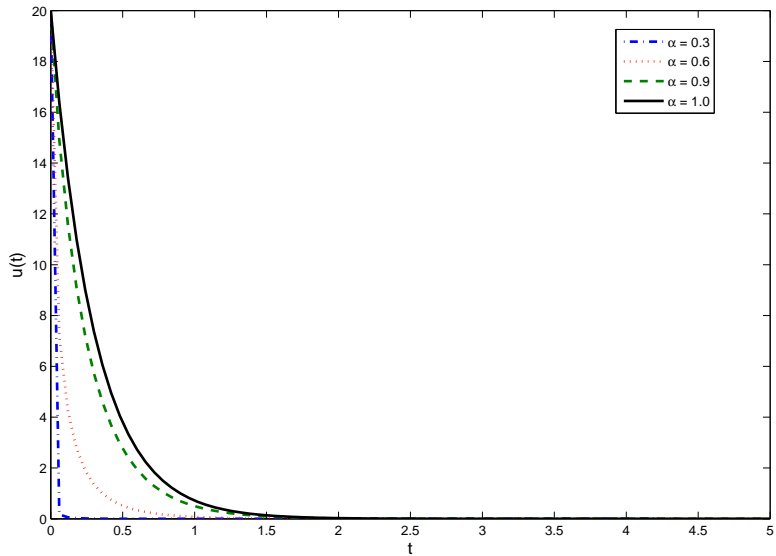} 
\label{fig:eryt1}
\end{figure}

\section{Concluding remarks} 

We introduced a new derivative, the local $M$-derivative, and its corresponding $M$-integral. We could prove important results concerning integer order derivatives of this kind, in particular, derivatives
of order one. For $\alpha$-differentiable functions in the context of local $M$-derivatives we could show that the derivative proposed here behaves well with respect to the product rule, the quotient rule, composition of functions and the chain rule. The local $M$-derivative of a constant is zero, differently from the case of the Riemann-Liouville fractional derivative.  Moreover, we present $\alpha$-differentiable functions versions of Rolle's theorem, the mean value theorem and the extended mean value theorem.

An $M$-integral was introduced and some results bearing relations to results in the calculus of integer order were obtained, among which the $M$-fractional versions of the inverse theorem, the fundamental theorem of calculus and a theorem involving integration by parts.

We obtained a relation between our local $M$-derivative and the alternative fractional derivative, presented in section 5 of the paper, as well as possible applications in several areas, particularly as we show, in the solution of a linear differential equation. We conclude from this result that the definition presented here can be considered a generalization of the so-called alternative fractional derivative \cite{UNK}.

Possible applications of the local $M$-derivative and the corresponding $M$-integral are the subject of a forthcoming paper \cite{JVC}.  

\section*{Acknowledgment} We are grateful to Dr. J. Em\'{\i}lio Maiorino 
for several and fruitful discussions.

\bibliography{ref}
\bibliographystyle{plain}

\end{document}